\documentclass[11pt]{amsart}

\usepackage{amsmath,amsfonts,amssymb,amsthm}
\usepackage{txfonts}
\usepackage{latexsym,bm,graphicx}
\usepackage{mathrsfs}
\usepackage{color}

\title[Liouville theorems]{Liouville theorems for semilinear differential inequalities on sub-Riemannian manifolds}

\author{Bing Wang}
\address{Bing Wang: Department of Mathematics, Sun Yat-sen University,  No. 135, Xingang Xi Road, Guangzhou, 510275 \\}
\email{wangb265@mail2.sysu.edu.cn}
\author{Hui-Chun Zhang}
\address{Hui-Chun Zhang: Department of Mathematics, Sun Yat-sen University, No. 135, Xingang Xi Road, Guangzhou, 510275 \\}
\email{zhanghc3@mail.sysu.edu.cn}

\newtheorem{theorem}{Theorem}[section]
\newtheorem{proposition}[theorem]{Proposition}
\newtheorem{lemma}[theorem]{Lemma}

\newtheorem{thm}{Theorem}
\newtheorem{hypothesis}{Hypothesis }

\theoremstyle{definition}
\theoremstyle{remark}

\newtheorem{definition}[theorem]{Definition}
\newtheorem{remark}[theorem]{Remark}

\numberwithin{equation}{section}

\newcommand{\ls}{\leqslant}
\newcommand{\gs}{\geqslant}
\newcommand{\dx}{\ {\rm d}x}
\newcommand{\dt}{\ {\rm d}t}
\newcommand{\ds}{\ {\rm d}s}
\newcommand{\drho}{\ {\rm d}\rho}
\newcommand{\dmu}{\ {\rm d}\mu}

\newcommand{\ip}[2]{\langle{#1},{#2}\rangle}

\newcommand{\R}{\mathbb{R}}
\newcommand{\G}{\mathbb{G}}
\newcommand{\cH}{\mathcal{H}}
\newcommand{\CD}{CD$(\rho_1, \rho_2, \kappa, d)$}

\begin{document}

\begin{abstract}
    In this paper, we generalize Liouville type theorems for some semilinear partial differential inequalities to 
    sub-Riemannian manifolds satisfying a nonnegative generalized curvature-dimension inequality introduced by 
    Baudoin and Garofalo in \cite{BG2017}.
    In particular, our results apply to all Sasakian manifolds with nonnegative horizontal Webster-Tanaka-Ricci curvature.
    The key ingredient is to construct a class of ``good" cut-off functions. 
    We also provide some upper bounds for lifespan to parabolic and hyperbolic inequalities.
\end{abstract}

\maketitle

\section{Introduction}
In the past two decades, the geometric analysis  aspects of non-Riemannian manifolds have many achievements. For instances, 
    the Liouville theorem for harmonic (or super-harmonic) functions have been discovered for Alexandrov spaces 
    \cite{ZZ2012,HX2014}, graphs \cite{HJL2015}, metric measure spaces \cite{CJKS2020, JKY2014, AGS2017} and sub-Riemannian manifolds 
    \cite{BG2017, BBG2014, CKLT2019}. The purpose of the present paper is to study the Liouville type theorems for solutions of some semilinear partial 
    differential inequalities  on sub-Riemannian manifolds, 
    including the subelliptic, parabolic and hyperbolic cases.
     
    Recalling that in the Euclidean space $\R^n$, probably the most outstanding result of this type is attached to the equation
    \begin{equation}\label{eq1.1}
        \Delta u+u^p=0 \ \ {\rm in} \ \ \R^n,\quad n>2,
    \end{equation}
    where $p>0$, and was obtained in \cite{Gidas}. It was shown there that (\ref{eq1.1}) admits no nontrivial non-negative solution 
    if and only if $1\ls p < \frac{n+2}{n-2}$. 
    Moreover, there exists no nontrivial nonnegative super-solution of (\ref{eq1.1}) provided that
    $$1\ls p\ls\frac{n}{n-2}.$$
    Fujita \cite{Fujita}  and Kato \cite{Kato} established the similar Liouville type theorems for the parabolic equation and the hyperbolic 
    equation on $\R^n$. They shown that (\cite{Fujita} for the parabolic case)
    $$\partial_t u-\Delta u\gs u^p  \ \ {\rm in} \ \ \R^+ \times \R^n $$
    with the initial data $u(0,x)=u_0(x)\in L^1(\R^n)$ admits no nontrivial nonnegative solution if 
    $$1<p\ls p_{\rm Fuj}:=1+\frac{2}{n} \quad {\rm and}\quad \liminf_{R\to\infty}\int_{B_R(0)}u_0 \dx \gs 0.$$
    and that (\cite{Kato} for the hyperbolic case)   
    $$\left\{
        \begin{aligned}
            &\partial_{tt}u-\Delta u\gs |u|^p, \quad \ {\rm in} \ \R^+ \times \R^n,    \\
            &u|_{t=0}=u_0,   \\  
            &\partial_t u|_{t=0}=u_1,
        \end{aligned}
    \right.$$
    admits no nontrivial global solution if 
    $$1<  p_\mathrm{Kato}:=\frac{n+1}{n-1}\quad {\rm and}\quad \liminf_{R\to\infty}\int_{B_R(0)}u_1 \dx \gs 0.$$

    Such Liouville type theorems have been widely generalized to more general elliptic operators on Euclidean spaces or Riemannian manifolds with vary different assumptions, we refer to \cite{SZ02,AS2014, GSV2020, GSXX2020, Z2015} (for the elliptic case),  \cite{MMP2017,Zqs1999} (for the parabolic case), \cite{MPS2020} (for the hyperbolic case) and their references. 

    On the aspect of non-Riemannian settings, such Liouville type theorems was first considered on Heisenberg groups. 
    Garofalo and Lanconelli \cite{GL1992} and later Birindelli, Capuzzo Dolcetta and Cutri \cite{BCC1997} obtained some 
    nonexistence results for positive solutions of nonlinear elliptic inequalities of the type
    $$\Delta_{\mathbb{H}^n}u+ |\eta|^\gamma_{\mathbb{H}^n}u^p\ls 0,\qquad \gamma>-2, \ \ p>0.$$
    S. Pohozaev and L. V\'eron \cite{PV2000} provide a systematic study to nonexistence results for 
    solutions of partial differential
    inequalities with nonlinear term
    $|\eta|^\gamma_{\mathbb{H}^n}|u|_{\mathbb{H}^n}^p$ on Heisenberg group $\mathbb{H}^n$. In particular, they proved the following results.
    \begin{thm}[Pohozaev-V\'eron \cite{PV2000}]
        Let $\mathbb{H}^n$ be the $(2n+1)$-dimensional Heisenberg group, and let $\Delta_{\mathbb{H}^n}$ be the canonical sub-Laplacian. 
        $Q:=2n+2$ is the homogeneous dimension of $\mathbb{H}^n$. 
        Then 
        \begin{itemize}
            \item[(1)] Assume that $1<p\ls Q/(Q-2)$, then there exists no nontrivial solution to 
                $$\Delta_{\mathbb{H}^n}u+|u|^p\ls 0\quad {\rm in}\  \ \mathbb{H}^n;$$ 
            \item[(2)]Assume that 
                $$1<p\ls(Q+2)/Q \ \  and \ \ \int_{\mathbb{H}^n}u_0(\eta)\mathrm{d}\eta \gs 0$$
                Then there exists no global in time nontrivial solution of
                \begin{equation*}
                    \left\{
                    \begin{aligned}
                    &\partial_t u-\Delta_{\mathbb{H}^n}u \gs |u|^p\quad {\rm in}\ \ \R^+\times\mathbb{H}^n \\
                    &u|_{t=0}=u_0\in L^1_{\rm loc}(\mathbb{H}^n); 
                    \end{aligned}
                    \right.
                \end{equation*}
            \item[(3)]Assume that
                $$1<p\ls(Q+1)/(Q-1) \ \  and \ \ \int_{\mathbb{H}^n}u_1(\eta)\mathrm{d}\eta \gs 0$$
                Then there exists no global in time nontrivial solution of
                \begin{equation*}
                    \left\{
                    \begin{aligned}
                    &\partial_{tt} u-\Delta_{\mathbb{H}^n}u \gs |u|^p,   \\
                    &u|_{t=0}=u_0, \\
                    &\partial_t u|_{t=0}=u_1,   
                    \end{aligned}
                    \right.
                \end{equation*}
                in $\R^+\times \mathbb{H}^n$, where $u_0,u_1\in L^1_{\rm loc}(\mathbb{H}^n).$
        \end{itemize}
    \end{thm}
    For more related works on this type of nonexistence theorems on Heisenberg groups, we refer to \cite{BarGo2020,Xu2009,Yu2013}, 
    for H-type groups \cite{BU2004}, for  
    Carnot groups \cite{AM2013,KTT2020}, and for metric measure spaces \cite{BJ2020}. 

    In this present work, we provide a systematic study of this type of Liouville theorems on more general sub-Riemannian manifolds. 
    To state our main results we now recall the relevant framework.  Let $M$ be a $C^\infty$ sub-Riemannian manifold with sub-Riemannian structure $(M,\mathcal H,g)$ endowed with
    a $C^\infty$ measure $\mu$, where $\cH$ is a bracket-generating 
        distribution  by  the vector fields $\{X_i\}_{1\ls i\ls d}$.  Let $L$ be a $C^\infty$ second-order diffusion, locally subelliptic operator 
    on $M$ with real coefficients.   Locally in the neighborhood of each $x\in M$ the operator $L$ can be written as 
        $$L=-\sum_{i=1}^d X_i^* X_i,$$
        where  $X_i^*$ are the formal adjoint of $X_i$ with respect to $\mu$.  Suppose that  
    \begin{equation}
        L1=0;\quad  \int_M fLg{\rm d}\mu=\int_MgLf{\rm d}\mu;\quad\int_MfLf{\rm d}\mu\ls0
    \end{equation} 
    for any $ f, g \in C^\infty_0(M)$.  
    The $carr\acute{e}\ du\ champ$ 
    $\Gamma(f)$ associated to $L$ is defined by
    $$\Gamma(f)=\Gamma(f,f), \quad \Gamma(f,g)=\frac{1}{2}(L(fg)-fLg-gLf), \quad f,g\in C^\infty(M).$$
    The following distance is canonically associated with the operator $L$, which is given by:
    \begin{equation}
        d_L(x, y) = \sup_{}
        \big\{| f (x)- f (y)| \ |\ f \in C^\infty(M), \
        \|\Gamma(f)\|_{\infty}\ls 1\big\}.  
    \end{equation}
    Throughout this paper the metric space $(M,d_L)$ is assumed to be complete.
  
    As in \cite{BG2017}, we also assume  that $M$ is equipped with another first-order differential symmetric bilinear form 
    $\Gamma^Z: C^\infty(M)\times C^\infty(M)\to \R$ such that $\Gamma^Z(f,g)=\Gamma^Z(g,f)$ and $\Gamma^Z(fg,h)=f\Gamma^Z(g,h)+g\Gamma^Z(f,h)$. 
    Analogous to $\Gamma$, we assume that $\Gamma^Z(f)\gs 0$, where $\Gamma^Z(f)=\Gamma^Z(f,f)$, and that    
    $$\Gamma^Z(\rho\circ\varphi)\ls\left(\rho'(\varphi)\right)^2\Gamma^Z(\varphi),\quad \forall\  \rho\in C^\infty(M),\ \varphi\in C_0^\infty(M).$$ 
    Baudoin and Garofalo \cite{BG2017}
    introduced a generalized curvature-dimension inequality for sub-Riemannian manifolds as follows:

    \begin{definition}[Baudoin-Galofalo \cite{BG2017}]\label{defn1}
        $M$ is said to be satisfying the \emph{generalized curvature-dimension inequality} \CD \ with respect to $L$ and $\Gamma^Z$
        if there exists constants $\rho_1\in\R$, $\rho_2\gs 0$, $\kappa\gs 0$ and $0<d\ls\infty$ such that
        $$\Gamma_2(f)+\nu\Gamma_2^Z(f)\gs\frac{1}{d}(Lf)^2+\left(\rho_1-\frac{\kappa}{\nu}\right)\Gamma(f)+\rho_2\Gamma^Z(f)$$
        for every $f\in C^\infty(M)$ and every $\nu>0$, where
        $$\Gamma_2(f,g)=\frac 1 2\big(L\Gamma(f,g)-\Gamma(f,Lg)-\Gamma(g,Lf)\big),\ \ \ $$
        $$\Gamma^Z_2(f,g)=\frac 1 2\big(L\Gamma^Z(f,g)-\Gamma^Z(f,Lg)-\Gamma^Z(g,Lf)\big), $$
        $\Gamma_2(f)=\Gamma_2(f,f)$ and  $\Gamma^Z_2(f)=\Gamma^Z_2(f,f)$.
    \end{definition}
    
When M is a complete Riemannian manifold,
    \CD \ is a generalization of the Bakry-Emery curvature-dimention condition $ \mathrm{CD}(\rho_1,d)$ with respect to $L=\Delta$
    the Belmetri-Laplace operator and $\Gamma^Z(f)=0$ for all $f\in C^\infty_0(M)$. 
    In general, the choice of $\Gamma^Z$ is not a priori canonical. However, in the main geometrical examples given in \cite{BG2017}, 
    the choice of $\Gamma^Z$ is canonical. For example, for CR Sasakian manifolds $M$, one can consider the gradient of a function 
    $f\in C^\infty(M)$ to have the horizontal part $\nabla_Hf$ and the vertical part $\nabla_Vf$. 
    In this case, we choose $\Gamma^Z(f)=|\nabla_Vf|^2.$
    
    As in \cite{BG2017}, we make the following four basic assumptions:

    \begin{hypothesis} \label{a1}
        There exists an increasing sequence $h_k\in C_0^\infty(M)$ such that $h_k\nearrow1$ on $M$, and
        $$\|\Gamma(h_k)\|_\infty+\|\Gamma^Z(h_k)\|_\infty\to 0 \quad as\ k\to 0.$$  
    \end{hypothesis}

    \begin{hypothesis} \label{a2} 
        For any $f\in C^\infty(M)$, one has
        $$\Gamma(f,\Gamma^Z(f))=\Gamma^Z(f,\Gamma(f)).$$
    \end{hypothesis}

    \begin{hypothesis} \label{a3}
        Given any two points $x, y \in M$, there exists a subunit curve joining them (see \cite{BG2017} or Sect. 2 for the definition of subunit curve). 
        It is for instance fulfilled when the operator $ L$ satisfies the
        finite rank condition of the Chow-Rashevsky theorem (see \cite{BG2017}).
    \end{hypothesis}

    \begin{hypothesis} \label{a4}
        The heat semigroup generated by $L$ is stochastically complete, i.e. for $t\gs 0$, $P_t(1)=1$ and for every 
        $f\in C_0^\infty(M)$, $T\gs 0$, 
        $$\sup_{t\in[0,T]}\|\Gamma(P_t(f))\|_\infty+\|\Gamma^Z(P_t(f))\|_\infty<\infty.$$  
    \end{hypothesis}


    The first purpose of this paper is to extend the the Liouville theorems in \cite{PV2000} from Heisenberg groups to more general 
    sub-Riemannian manifolds. This is the content of the following result.
    \begin{theorem}\label{thm1.2}
        Let $(M,\mu)$ be a sub-Riemannian manifold, and let $L$ be a subelliptic diffusion operator. Suppose that there is a bilinear form 
        $\Gamma^Z$ satisfying the above Hypothesis {\rm (\ref{a1}-\ref{a4})} and that  $M$ satisfies the generalized curvature-dimension inequality 
        $CD(0,\rho_2,\kappa,d)$ \ with respect to $L$ and $\Gamma^Z$ for some $\rho_2>0,\kappa\gs 0$ and $0<d\ls \infty.$
        Set 
        \begin{equation}\label{equ1.4}
        D=\left(1+\frac{3\kappa}{2\rho_2}\right)d.
        \end{equation}
        Then the following results hold:
        \begin{itemize}
            \item[(1)]
                The subelliptic inequality
                $$ -Lu\gs |u|^p,$$ 
                admits no nontrivial solution provided that 
                $$1<p\ls\frac{D}{D-2}.$$
            \item[(2)]
                The parabolic inequality
                \begin{equation}  \label{intro1}
                    \left\{
                        \begin{aligned}
                            &\partial_t u(t,x)-Lu \gs |u|^p, x\in M,t>0,    \\
                            &u(0,x)= u_0(x),     
                        \end{aligned}
                    \right.
                \end{equation}
                with 
                $$\liminf_{R\to\infty}\int_{B_R(x_0)}u_0(x)\dmu(x)>0,$$
                admits no global in time solution provided that $1<p\ls 1+\frac{2}{D}$.
            \item[(3)]  
                The hyperbolic inequalities
                \begin{equation} \label{intro2}
                    \left\{
                        \begin{aligned}
                            &\partial_{tt} u-Lu \gs |u|^p, x\in M,t>0,    \\
                            &u(0,x)=u_0(x),   \\
                            &\partial_t u(0,x)=u_1(x),   
                        \end{aligned}
                    \right.
                \end{equation}
                with 
                $$\liminf_{R\to\infty}\int_{B_R(x_0)}u_1(x)\dmu(x)>0,$$
                admits no global in time solution provided that $1<p\ls\frac{D+1}{D-1}$.
        \end{itemize}
    \end{theorem}

    The CR Sasakian manifolds are of the main examples for sub-Riemannian manifolds which satisfy  the generalized curvature-dimension conditions in Definition \ref{defn1}. The following Liouville theorems  are of independent interests.

    \begin{theorem}\label{thm1.3}
        Let $(M,\theta)$ be a complete Sasakian manifold with real dimension $2n+1$. If the Tanaka-Webster Ricci curvature on $M$ is nonnegative, by letting  $D_n=2n+3$,
        then the following Liouville properties hold:
        \begin{itemize}
            \item[(1)]
                The subelliptic inequality
                $$ -Lu\gs |u|^p,$$ 
                admits no nontrivial solution provided that 
                $$1<p\ls\frac{D_n}{D_n-2}.$$
            \item[(2)]
                The parabolic inequality (\ref{intro1})
                with 
                $$\liminf_{R\to\infty}\int_{B_R(x_0)}u_0(x)\dmu(x)>0,$$
                admits no global in time solution provided that $1<p\ls 1+\frac{2}{D_n}$.
            \item[(3)]  
                The hyperbolic inequalities (\ref{intro2})
                with 
                $$\liminf_{R\to\infty}\int_{B_R(x_0)}u_1(x)\dmu(x)>0,$$
                admits no global in time solution provided that $1<p\ls\frac{D_n+1}{D_n-1}$.
        \end{itemize}

    \end{theorem}

    For parabolic and hyperbolic equations with small initial data in the Euclidean case, 
    a quantitative estimate for lifespan time has been given in \cite{LN} and \cite{TY}. More recent results and detailed survey can 
    be find in \cite{Ikeda}. 
    This kind of results have been generalized to Heisenberg groups and Carnot groups in \cite{GPHei},
    \cite{GPCarnot}, respectively.
    
    In this paper we will extend such estimates to sub-Riemannian settings in subcritical case.

    \begin{theorem} \label{thm1.4}
        Let $(M,\mu)$ be a sub-Riemannian manifold satisfying Hypothesis {\rm (\ref{a1}-\ref{a4})},  $L$ is a sub-Laplacian on $M$.
        Assume that $M$ satisfies CD$(0,\rho_2,\kappa,d)$ and let $D$ be given in {\rm (\ref{equ1.4})}.
        \begin{itemize}
            \item[(1)]Assume that $$1<p< 1+\frac{2}{D} \quad {and}\quad \liminf_{R\to\infty}\int_{B_R}u_0(x)\dmu(x)>0.$$
               if $u(t,x): [0,T)\times M\to \R$ is a solution to
                \begin{equation*} 
                    \left\{
                        \begin{aligned}
                            &\partial_t u(t,x)-Lu \gs |u|^p, x\in M,t>0,    \\
                            &u(0,x)= \varepsilon u_0(x),\quad \varepsilon>0,
                        \end{aligned}
                    \right.
                \end{equation*} 
           then
             \begin{equation}\label{equ1.7}
             \begin{split}
             T\ls T(\varepsilon):=C\varepsilon^{-(\frac{1}{p-1}-\frac{D}{2})^{-1}}.
             \end{split}
             \end{equation}
            \item[(2)]Assume that $$1<p<\frac{D+1}{D-1}\quad and\quad \liminf_{R\to\infty}\int_{B_R}u_1(x)\dmu(x)>0.$$
               if  $u(t,x): [0,T)\times M\to \R$ is a solution to 
                \begin{equation*}
                    \left\{
                        \begin{aligned}
                            &\partial_{tt} u-Lu \gs |u|^p, x\in M,t>0,    \\
                            &u(0,x)=u_0(x),   \\
                            &\partial_t u(0,x)=\sigma u_1(x), \quad \sigma>0,  
                        \end{aligned}
                    \right.
                \end{equation*}
             then
              \begin{equation}\label{equ1.8}
              T\ls T(\sigma):= C\sigma^{-(\frac{p+1}{p-1}-D)^{-1}}.
              \end{equation}
        \end{itemize}
    \end{theorem}

    \begin{remark}\label{rem1.5}
        In particular, Theorem \ref{thm1.4} apply to all Sasakian manifolds with nonnegative horizontal Tanaka-Webster Ricci curvature and Carnot groups of step 2.
    \end{remark}
 
      For parabolic differential inequalities,    the estimate   in Theorem \ref{thm1.4}(1) has been proven   in \cite{GPCarnot} for Carnot groups of 
        arbitrary step.  For the hyperbolic case,  Theorem\ref{thm1.4}(2)  implies that the  the lifespan estimate   in (\ref{equ1.8})  holds for Carnot groups of step 2.  So, at last, we  will establish an lifespan estimate for the hyperbolic differential inequalities  on Carnot groups of arbitrary step, independent of Theorem \ref{thm1.4}.

    \begin{theorem}\label{thm1.7}
        Let $\G$ be a Carnot group of arbitrary step, and $Q$ be its homogeneous dimension. Assume $1<p\ls \frac{Q+1}{Q-1}$. Then there exists a constant $\sigma_0>0$ such that for any $\sigma\in(0,\sigma_0)$ and any solution $u:[0,T)\times M\to \R$ of  
        \begin{equation*}
            \left\{
                \begin{aligned}
                    &\partial_{tt}u-\Delta_\G u \gs |u|^p, x\in \G,t>0,    \\
                    &u(0,x)=u_0(x),   \\
                    &\partial_tu(0,x)=\sigma u_1(x),  \sigma>0,    
                \end{aligned}
            \right.
        \end{equation*}
        with $u_1$ satisfying
        \begin{equation*} 
            \liminf_{R\to\infty}\int_{B_R(0)}u_1(x)\dmu(x)>0,
        \end{equation*}
        we have    
        \begin{equation}
            T(\sigma)\ls
            \begin{cases}
                C\sigma^{-({\frac{p+1}{p-1}-Q)}^{-1}} & if\ p\in(1,\frac{Q+1}{Q-1}), \\
                \exp(C\sigma^{-(p-1)}) & if\ p=\frac{Q+1}{Q-1},
            \end{cases}
        \end{equation}
        where $C$ is a positive constant independent of $\sigma$.
    \end{theorem}

    This paper is organized as follows. In Section 2, we recall some facts about sub-Riemannian manifolds with transverse symmetry.
    Then in Section 3, we construct the cut-off functions on sub-Riemannian manifolds. In Section 4, we prove the nonexistence results then use them to drive the upper bounds of lifespan time for parabolic and 
    hyperbolic inequalities in Section 5. In the last section, we independently prove the upper bounds of lifespan time for 
    hyperbolic inequalities on Carnot group with arbitrary step.
\\
  
{\bf Acknowledgments}. This work was supported by the National Natural Science Foundation of China (NSFC) (No. 11521101 and 12025109).

\section{Preliminaries}

    \subsection{The  sub-Riemannian geometry and the generalized curvature-dimension inequality} $\ $ \label{Pre}
    
        In this subsection we will briefly recall the notations on sub-Riemannian manifolds  and some consequences of generalized curvature-dimension inequalities introduced in \cite{BG2017}.
        
      Let   $(M,g_M)$ be a $(d+\mathfrak{h})$-dimensional Riemannian manifold equipped with a sub-Riamannian structure $(\cH,g:=g_M|_{\mathcal H})$ with transverse symmetries in the sense of Baudoin-Garofalo in \cite{BG2017}. 
        That is,  $\cH$ is a bracket-generating 
        distribution of dimension $d$, and $\mathcal V$ is a vertical distribution of  dimension $\mathfrak h$ such that the decomposition $T_xM=\cH(x)\oplus\mathcal V(x)$ is orthogonal with respect to $g_M$ at every point $x\in M$.   The distribution $\cH$ (and $\mathcal V$ resp.) will be referred to as the set of $horizontal$ (and $vertical$ resp.) directions.  Now the Riemannian metric 
        $$g_M=g\oplus g_V,$$
        where $g_V=g_M|_{\mathcal V}$ is the induced by $g_M$ onto the vertical bundle $\mathcal V$.
        Note that now Hypothesis (H\ref{a1}) is equivalent to that $(M,g_M)$ is a complete Riemannian manifold. 
        The horizontal gradient $\nabla_Hf$ (and the vertical gradient $\nabla_Vf$ resp.) of a function $f$ is given by the projections
        of the gradient of $f$ via $g_M$ on the horizontal bundle (and vertical bundle resp.). 
        We refer to Sect. 2.3 of \cite{BG2017} for the details.       
  
        Let $\mu$ be the Riemannian measure of $(M, g_M)$. 
        The canonical sub-Laplacian in this structure is the diffusion operator $L$ which is symmetric with respect to $\mu$
        and its $carr\acute{e}\ du\ champ$ is given by
\begin{equation}
\begin{split}
\Gamma(f_1,f_2)&=\frac 1 2(L(f_1f_2)-f_1Lf_2-f_2Lf_1)\\
&=g_M(\nabla_Hf_1,\nabla_Hf_2),\qquad\qquad f_1,f_2\in C^\infty_0(M).
\end{split}
\end{equation}
        It is known in \cite{BG2017} that this operator $L$ satisfies (1.2) in Introduction. Suppose that $d_L(x,y)$ is 
        the metric canonically associated with $L$ and $\Gamma$ given in (1.3) in Introduction. This distance $d_L(x,y)$ can be given as follows. An absolutely continuous curve $\gamma:[0,T]\to M$ is called subunit  for the operator $\Gamma$ if for  almost all $t\in[0,T]$ one has $\gamma'(t)\in \mathcal H(\gamma(t))$ and  $|\frac{d}{dt}f\circ \gamma(t)|^2\ls \Gamma(f)(\gamma(t))$
for every smooth function $f$ on $M$.  For each pair $(x,y)$, the set
$$S(x,y):=\big\{\gamma:[0,T]\to M\ |\ \gamma \ {\rm is\ subunit\ and }\ \gamma(0)=x,\ \gamma(T)=y\big\}.$$
Since the  distribution $\mathcal H$ is bracket-generated,  by the Chow-Rashevsky theorem,  one has  $S(x,y)\not=\emptyset$ for any $x,y\in M$. Hence the Hypothesis (H\ref{a3}) holds. It is known that $d(x,y)$ in (1.3) can be also given as
 $$d_L(x,y)=\inf_{\gamma\in S(x,y)} \int_0^T|\gamma'(t)|_gdt.$$

Recalling that, locally in the neighborhood of each $x\in M$ the sub-Laplacian $L$ can be written as 
        $$L=-\sum_{i=1}^d X_i^* X_i,$$
        where the vector fields $X_i$ are bracket-generators (a base of $\cH$) and  $X_i^*$ are the formal adjoint of $X_i$ with respect to $\mu$. Direct computation shows that
        $$L=\sum_{i=1}^d X_i^2+\mathrm{div}_{\mu}(X_i)X_i,$$
        where $\mathrm{div}_{\mu}(X_i)$ is defined by the identity
        $$\mathcal{L}_{X_i}\mu=\mathrm{div}_{\mu}(X_i)\mu,$$
        in which $\mathcal{L}_{X_i}$ denots the Lie derivative in the direction $X_i$.
        Since $\cH$ is bracket-generating,  it is well-known that 
        $L$ is hypoelliptic.

        The completeness of $(M,d_L)$ as a metric space guarantees the uniqueness of the solution to the Cauchy problem
        \begin{equation} \label{heat}
            \left\{
                \begin{aligned}
                    &\partial_tu-Lu=0  \\
                    &u|_{t=0}=f\in L^2(M,\mu)
                \end{aligned}
            \right.
        \end{equation}
        on $[0,\infty)\times M$. And by the hypoellipcity of $L$, such a solution is $C^\infty$ in $(0,\infty)\times M$.
        Moreover, the unique solution to (\ref{heat}) on $(0,\infty)\times M$ can be written as
        $$u(t,x)=\int_M p(x,y,t)f(y)\dmu(y),$$
        where $p(\cdot,\cdot,t)$ is the co-called heat kernel, which is smooth and symmetric on $(0,\infty)\times M \times M$.
        The heat semigroup $(P_t)_{t\gs 0}$ associated to $L$ can be defined as
        $$P_tf(x)=\int_M p(x,y,t)f(y)\dmu(y), \quad \mathrm{for} \ f\in L^2(M,\mu),$$
        Note that $(P_t)_{t\gs 0}$ is a sub-Markov semigroup, i.e. $P_t1\ls 1.$

        The bilinear form $\Gamma^Z$ on a sub-Riemannian manifold $M$ with transverse symmetries  
        can be chosen canonically by \cite{BG2017} 
        \begin{equation}
            \Gamma^Z(f,g)=g_M(\nabla_Vf_1,\nabla_Vf_2), \quad \forall f_1,f_2\in C^\infty_0(M).
        \end{equation}
        It was proved in \cite{BG2017} that  $\Gamma^Z$ satisfies Hypothesis (H\ref{a2}) in Introduction. 
        The  combination of the completeness of $g_M$ and the Nagel-Stein-Wainger theorem \cite{NSW1985} ensures that Hypothesis (H\ref{a1}) is fulfilled.

 At last, it was shown in \cite[Theorem 4.3]{BG2017} that the Hypothesis (H\ref{a4}) follows from the generalized curvature-dimension inequality \CD \  in the definition in Introduction on Sasakian manifolds and Carnot groups of step two. In the following two subsections, we will  recall some facts on these two classes of important geometrical  examples.

        \subsection{CR Sasakian manifolds} $\ $
        
        Let $(M,\theta)$ be a complete  strictly pseudo-convex CR Sasakian manifold with $(2n+1)$  real dimension, where $\theta$ is a pseudo-hermitian form on $M$ such that its Levi form is positive definite.  The kernel of $\theta$ provides a horizontal bundle $\cH$. Let $Z$ be the Reeb vector field with respect to $\theta$.
  Let $\nabla$ be the Tanaka-Webster connection of $M$. We recall that the CR manifold $(M, \theta)$ is called Sasakian if the pseudo-hermitian
torsion of $\nabla $ vanishes. For example,
the standard CR structures on the $(2n+1)$-dimensional Heisenberg group $\mathbb H^n$ and the sphere $\mathbb S^{2n+1}$ are
Sasakian. 

In every Sasakian manifold $(M,\theta)$, the Reeb vector field $Z$ provides a vertical vector field. Let $g_1$ be an extension of Levi form such that $g_1(X,Z)=0$ and $g_1(Z,Z)=1$ for all $X\in \cH$. Now the canonical choice of $\Gamma^Z$ is 
$$\Gamma^Z(f)=|Zf|_{g_1},\qquad \forall f\in C^\infty_0(M).$$
\begin{theorem}[Theorem 1.7 of \cite{BG2017}]
Let $(M,\theta)$ be a complete Sasakian manifold with real dimension $2n+1$. If for every $x\in M$, the Tanaka-Webster Ricci tensor satisfies the bound
$$Ric_x(v,v)\gs \rho_1|v|^2$$
for every horizontal vector $ v\in\cH_x$, then for the CR sub-Laplacian of $M$ the curvature dimension inequality $\mathrm{CD}(\rho_1,n/2,1,2n)$ holds and Hypothesis {\rm (\ref{a1}-\ref{a4})} are satisfied. 
\end{theorem}

         \subsection{Carnot groups} \label{CarnotGroups}$ \ $
         
            Let $(\G,\ast)$ be a connected and simply connected Lie group whose Lie algebra $\mathfrak{g}$ admits a stratification, 
            i.e. a direct sum decomposition
            \begin{equation} \label{stratification}
                \mathfrak{g}=V_1\oplus V_2\oplus\cdots\oplus V_r \quad \mathrm{with} \quad
                \left\{
                    \begin{aligned}
                        &[V_1,V_{i-1}]=V_i, \quad \mathrm{if} \  2\ls i\ls r, \\
                        &[V_1,V_r]=0.
                    \end{aligned}
                \right.
            \end{equation}
            Then $\G$ is called a {\emph{Carnot group (or stratified Lie group) of step $r$}}. 
     Let $n_i=\dim V_i$ for all $1\ls i\ls r$ and  $n=\sum_{i=1}^r n_i$,  which is  the dimension of the 
            vector space $\mathfrak{g}$.  The strata $V_1$ is called the horizontal layer of $\G$.

  Since $\G$ is simply connected, the exponential map $\exp: \mathfrak{g}\to \G$
   is a global diffeomorphism.  Given a basis 
            $$\mathcal{X}=\{X_1^{(1)},\cdots,X_{n_1}^{(1)};\cdots;X_1^{(r)},\cdots,X_{n_r}^{(r)}\}$$
            of $\mathfrak{g}$ adapted to the stratification in (\ref{stratification}), 
            i.e.  $\{X_1^{(i)},\cdots,X_{n_i}^{(i)}\}$ is a basis of $V_i$ for every $1\ls i\ls r$.
  The exponential map  enables us a global coordinate system on $\G$. We can identify the points of $ \G $ with whose coordinate  $x=(x^{(1)},\cdots,x^{(r)})\in \R^n$,where $x^{(i)}=(x_1^{(i)},\cdots,x_{n_i}^{(i)})\in \R^{n_i}$.

             The stratified structure of $\mathfrak{g}$ natrually introduce
            a family  of dialations $\{\delta_{\lambda}\}_{\lambda>0}$  on $\R^n$ of the following form:
            \begin{equation}\label{equ2.5}
            \delta_{\lambda}(x^{(1)},x^{(2)},\cdots,x^{(r)})=(\lambda x^{(1)},\lambda^2x^{(2)},\cdots,\lambda^rx^{(r)}).
            \end{equation}
        For each $1\ls k \ls n_i$,  $X_k^{(i)}$ is $\delta_\lambda$-homogeneous of degree $i$, i.e.
            $X_k^{(i)}(\varphi\circ\delta_\lambda)=\lambda^i X_k^{(i)}(\varphi)\circ\delta_\lambda$ for \ every \ $\varphi \in C^{\infty}(\G)\ \mathrm{and}\ \lambda>0.$
 
A homogeneous norm $|x|_{\G} $ is given by 
            \begin{equation}\label{equ2.6}
            |x|_{\G}:=\left(\sum_{i=1}^r|x^{(i)}|^{\frac{2r!}{i}}\right)^{\frac{1}{2r!}},
            \end{equation}
          where $|x^{(i)}|^2= \sum_{j=1}^{n_i}(x_{j}^{(i)})^2$ is  the norm on  $\R^{n_i}$ for each $i=1,\cdots, r.$
Remark that $|x|_{\G}$ is $\delta_\lambda$-homogeneous of degree 1 and that $|\cdot|_{\G}^{2r!}$ is a $C^\infty$ function.
            The distance defined by
            \begin{equation}\label{equ2.7}
            d_\G(x,y)=|y^{-1}\cdot x|_{\G}
            \end{equation}
            is equivalent to the Carnot-Carath\'eodory distance on $\G$ with respect to the horizontal layer $V_1$.
          
             Let $\mu$ be the Haar measure on $\G$, which is induced by the exponential map from the Lebesgue measure on $\mathfrak{g}$.   Then $\mu(\delta_\lambda(E))=\lambda^Q\mu(E)$ for any $\mu$-measurable  subset $E\subset \G$, where the number 
            $$Q:=\sum_{i=1}^ri\dim{V_i}$$
            is called the {\emph{ homogeneous dimension}} of $\G$. In particular,          
           $\mu(B^{\G}_R(x_0))=\omega_\G R^Q$, where $\omega_\G =\mu(B^{\G}_1(0))$ and 
          $B^{\G}_R(x_0):=\{x\in \G:d_\G(x,x_0)<R\}.$

 We denote $d=n_1$ the dimension of the horizontal layer of $\G$ and  write its basis   $\{X^{(1)}_1,\cdots,X^{(1)}_d\}$   as $\{X_1,\cdots,X_d\}$.        
     Then, under the Haar measure $\mu$, we have    $X_k^*=-X_k$ for each $1\ls k \ls d$, and the sub-Laplacian on $\G$ can be expressed as
            $$\Delta_{\G}=\sum_{k=1}^d X_k^2.$$
            
  If a Carnot group $G$ has step $r=2$, it provides a natural sub-Riemannian with transverse symmetries via $\mathfrak{g}=V_1\oplus V_2.$
We denote $\mathfrak{h}=n_2$ and $Z_j=X^{(2)}_j$ for all $j=1,\cdots, n_2$.

\begin{proposition}[Proposition 2.21 in \cite{BG2017}]
    Let $\G$ be a Carnot group of step two, and let $d$ be the dimension of the horizontal layer of its Lie algebra. Then $\G$ satisfies
    the generalized curvature-dimension inequality $\mathrm{CD}(0,\rho_2,\kappa,d)$ (with respect to any sub-Laplacian on $\G$) with 
    $$\rho_2=\inf_{\|z\|=1}\frac{1}{4}\sum_{i,j=1}^d \left(\sum_{m=1}^\mathfrak{h}\gamma_{ij}^m z_m\right)^2,$$
    $$\kappa=\sup_{\|x\|=1}\sum_{j=1}^d \sum_{m=1}^\mathfrak{h}\left(\sum_{i=1}^d\gamma_{ij}^m x_m\right)^2,$$ 
    where $\gamma_{ij}^m$ is the coefficients such that 
    $$[X_i,X_j]=\sum_{m=1}^\mathfrak{h}\gamma_{ij}^mZ_m.$$
\end{proposition}

\section{The construction of ``Good'' cut-off functions}
    In this section we construct the cut-off functions which are need in proving the main theorems of this paper. 
Let $(M,\mu)$ be a sub-Riemannian manifold and let $L$ be a sub-Laplacian on $M$. Suppose  that  $(M,\mu)$  satisfies \CD \ with respect to $L$ and  a bilinear form $\Gamma^Z$, for some $\rho_1\in\mathbb R$, $\rho_2\gs 0$, $\kappa\gs0$ and $d>0$.  We also assume that  the  Hypothesis (\ref{a1}-\ref{a4}) hold for $L$, $\Gamma^Z$ and  $(P_t)_{t\gs 0}$, which is the heat semigroup associated to $L$.

Let us begin from a distance comparison.    By using $\Gamma^Z$, a family of control distance $d_\tau$, $\tau\gs 0$, were introduced  in \cite{BBGM2014}. Given $x,y\in M$ and $\tau\gs0$, let
 $S_\tau(x,y)$ be the set of all curves $\gamma:[0,T]\to M$ jointed $x$ and $y$ such that  $|\frac{d}{dt}f\circ \gamma(t)|^2\ls (\Gamma(f)+\tau^2\Gamma^Z(f))(\gamma(t))$
for every smooth function $f$ on $M$.   It is clear that $S(x,y)\subset S_\tau(x,y)$ for any $\tau>0$. Define the distance
 $$d_\tau(x,y):=\inf_{\gamma\in S_\tau(x,y)} \int_0^T\ip{\gamma'(t)}{\gamma'(t)}_{g\oplus \tau^2g_V}dt.$$
In particular, when $\tau=1$, $d_1$ is the Riemannian distance induced by $g_R$ on $M$.  Note that, for all $\tau>0$,
$$  d_\tau(x,y)\ls d_0(x,y)=d_L(x,y),\qquad \forall\ x,y\in M.$$  

      \begin{proposition}[Baudoin el. \cite{BBGM2014}]\label{prop3.1}
        Let $(M,\mu)$ be a sub-Riemannian manifold satisfying Hypothesis (\ref{a1}-\ref{a4}), and $L$ be a sub-Laplacian on $M$.
        Assume that $M$ satisfies CD$(\rho_1,\rho_2,\kappa,d)$ for some $\rho_1\in \mathbb R$, $\rho_2>0$, $\kappa\gs 0$ and $0<d\ls \infty$.
        Then for any $\tau\gs0$, there exists a constant $C(\tau)>0$, depending only on $\rho_1,\rho_2,\kappa,d$, for which one have:
        $$d_L(x,y)\ls C(\tau)\cdot\max\big\{d_\tau(x,y), \sqrt{d_\tau(x,y)}\big\},\quad \forall\ x,y\in M.  $$
        Moreover, if $\rho_1\gs0$,  then one have  for any $\tau>0$ that
\begin{equation}\label{equ3.1}
 d_L(x,y)\ls C_{\rho_2,\kappa,d}\big(\tau+d_\tau(x,y)\big),\qquad \forall \ x,y\in M,
\end{equation}
where $C_{\rho_2,\kappa,d} $ is a  positive constant depending only on $\rho_2,\kappa$ and $d$.         \end{proposition}
\begin{proof}
The first assertion comes from \cite[Theorem 1.2]{BBGM2014}.

Let us consider  the second assertion (\ref{equ3.1}). From \cite[Line 13 on the page 2024]{BBGM2014}, by taking $\rho_1=0$ there, one get that 
$$d^2_L(x,y)\ls A_1t+A_2\cdot d^2_\tau(x,y)+A_3\frac{\tau^2d^2_\tau(x,y)}{t},\qquad x,y\in M $$
for all $t>0$, where $A_i, 1\ls i\ls 3$ are three positive constants depending only on $d,\kappa$ and $\rho_2.$ By choosing $t=\tau d_\tau(x,y)$, this yields 
\begin{equation*}
\begin{split}
d^2_L(x,y)&\ls \big(A_1+A_2+A_3\big)\left(\tau d_\tau(x,y)+ d^2_\tau(x,y)\right)\\
&\ls  \big(A_1+A_2+A_3\big)\big(\tau+ d_\tau(x,y)\big)^2,\qquad x,y\in M.
\end{split}
\end{equation*}
Let $C_{\rho_2,\kappa,d}=\sqrt{A_1+A_2+A_3}$ and the second inequality follows. The proof is finished. 
\end{proof}

 With the help of Proposition \ref{prop3.1}, we can improve  the Hypothesis {\ref{a1} to the following quantitative one.

    \begin{lemma}  \label{lem3.2}
        Assume that $(M,\mu)$ satisfies CD$(0,\rho_2,\kappa,d)$ for some  $\rho_2>0$, $\kappa\gs 0$ and $0<d\ls \infty$, and the Hypothesis {\rm (1-4)}.
    Then there exists a constant $\gamma:=\gamma(\rho_2,\kappa,d)>1$ such that:   for any $x_0\in M$ and any $R\gs 1$, there exists a cut-off function $\psi_0\in C_0^\infty(M)$ such that
        \begin{itemize}
            \item $0\ls\psi_0\ls 1$ on $M$, $\psi_0\equiv 1$ on $B_R(x_0)$ and $\mathrm{supp}\ \psi_0\subset B_{\gamma R}(x_0)$,
            \item $\|\Gamma(\psi_0)+\left(\frac{R}{2}\right)^2\cdot\Gamma^Z(\psi_0)\|_{\infty} \ls \frac{C_0}{R^2}$,
        \end{itemize}
        where $C_0$ is a universal  constant,  independent of $R$ and $x_0$.
    \end{lemma}

    \begin{proof}
        Fix any  $\tau>0$ and let $d_\tau$ be the distance reduced by $\Gamma +\tau^2\Gamma^Z$.
       Let $B_R(x_0)$ and $B^\tau_R(x_0)$ be the balls with radius $R$ under the metric $d_L$ and $d_\tau$ respectively. 
        Since for any $x, x_0\in M$, from $d_L(x,x_0)\gs d_\tau(x,x_0)$, we get
        $$B_R(x_0)\subset B^\tau_R(x_0).$$ 
        By Proposition \ref{prop3.1}, for $x\in M$ such that $d_\tau(x,x_0)\ls 2R$, there exists a constant $C_{\rho_2,\kappa,d}>0$ such that  
        $$d_L(x,x_0)\ls C_{\rho_2,\kappa,d}(\tau+d_\tau(x,y)) \ls \gamma\cdot R$$
        provided $\tau\ls R$ and where $\gamma:=3C_{\rho_2,\kappa,d},$ and hence we get that
        $B^\tau_{2R}(x_0)\subset B_{\gamma R}(x_0)$  when $\tau\ls R$.
 
        Take $\tau=\frac{R}{2}$ and since $(M,d_\tau)$ is complete, there exists a  cut-off function $\psi\in C_0^\infty(M)$ such that 
        \begin{itemize}
            \item $0\ls\psi\ls R$ on $M$, $\psi\equiv R$ on $B^\tau_R(x_0)$ and $\mathrm{supp}\ \psi\subset B^\tau_{2R}(x_0)$,
            \item $\Gamma(\psi)+\tau^2\Gamma^Z(\psi) \ls C_0$,
        \end{itemize}
     where $C_0>0$ is a universal constant.   Then by the fact that $B_R(x_0)\subset B^\tau_R(x_0)$ and $B^\tau_{2R}(x_0)\subset B_{\gamma R}(x_0)$, we have
        \begin{itemize}
            \item $0\ls\psi\ls R$ on $M$, $\psi\equiv R$ on $B_R(x_0)$ and $\mathrm{supp}\psi\subset B_{\gamma R}(x_0)$,
            \item $\Gamma(\psi)+\tau^2\Gamma^Z(\psi)\ls C_0(n) $.
        \end{itemize}
        Letting   $\psi_0=\frac{\psi}{R}$ and recalling $\tau=R/2$, the desired estimates follows, and the proof is finished.
    \end{proof}

The same argument gives the following cut-off functions in the case $\rho_1<0$. 
    \begin{remark}  \label{rem3.3}
        Assume that $(M,\mu)$ satisfies CD$(\rho_1,\rho_2,\kappa,d)$ for some  $\rho_1<0$, $\rho_2>0$, $\kappa\gs 0$ and $0<d\ls \infty$, and the Hypothesis {\rm (1-4)}.
    Then there exists a constant $\gamma:=\gamma(\rho_1,\rho_2,\kappa,d)>1$ such that:   for any $x_0\in M$ and any $R\gs 1$, there exists a cut-off function $\psi_0\in C_0^\infty(M)$ with
        \begin{itemize}
            \item $0\ls\psi_0\ls 1$ on $M$, $\psi_0\equiv 1$ on $B_R(x_0)$ and $\mathrm{supp}\ \psi_0\subset B_{\gamma R}(x_0)$,
            \item $\|\Gamma(\psi_0)+\Gamma^Z(\psi_0)\|_{\infty} \ls \frac{C_0}{R^2}$,
        \end{itemize}
        where $C_0$ is a universal  constant,  independent of $R$ and $x_0$.
    \end{remark}

   We also need the following  technical lemma  (c.f. \cite[Corollary 4.6]{BG2017}):
    \begin{lemma} \label{lem3.4}
        Assume that $(M,\mu)$ satisfies $\mathrm{CD}(\rho_1,\rho_2,\kappa,d)$ for some $\rho_1\in\mathbb R$, $\rho_2>0$, $\kappa\gs 0$ 
and $0<d\ls \infty$, and the Hypothesis {\rm (1-4)}.
    Then for all $f\in C_0^\infty(M)$, we have
        $$\Gamma(P_tf)+\nu\cdot\Gamma^Z(P_tf)+\frac{e^{-2At}-1}{-Ad}(L(P_tf))^2\ls P_t\left(\Gamma(f)+\nu\cdot\Gamma^Z(f)\right)e^{-2At},$$
        for any $\nu>0$, where  $A=A(\nu):=\min\{\rho_1-\kappa/\nu,\rho_2/\nu\}$.
        
        Here and in the sequel, we take $\frac{e^{-2At}-1}{-Ad}= \frac{2t}{d}$ when $A=0.$
            \end{lemma}

    \begin{proof}
        Set $\Psi(s)=P_s\left(\Gamma(P_{t-s}f)+\nu\cdot\Gamma^Z(P_{t-s}f)\right)$, as in \cite{BG2017}.
        
        Differentiating $\Psi(s)$ and by noting that $M$ satisfies $CD(\rho_1, \rho_2, \kappa, d)$ with respect to $L$ and $\Gamma^Z$, we have for all $\nu>0$ that
        \begin{equation*} 
            \begin{aligned}
                \Psi'(s)&=2P_s\left(\Gamma_2(P_{t-s}f)+\nu\Gamma_2^Z(P_{t-s}f)\right) \\
                &\gs 2P_s\left(\frac{1}{d}(L(P_{t-s}f))^2+(\rho_1-\kappa/\nu)\Gamma(P_{t-s}f)+\rho_2\Gamma^Z(P_{t-s}f)\right) \\
                  &\gs 2P_s\left(\frac{1}{d}(L(P_{t-s}f))^2+A\left(\Gamma(P_{t-s}f)+\nu\Gamma^Z(P_{t-s}f)\right)\right) \\
                &\gs \frac{2}{d}(L(P_tf))^2+2A\Psi(s),
            \end{aligned}
        \end{equation*}
        where we have used $P_s[(L(P_{t-s}f))^2]\gs [P_s (L(P_{t-s}f))]^2=(L(P_tf))^2$.     This yields 
                \begin{equation*}
            \begin{aligned}
                \left(e^{-2As}\Psi(s)\right)'&=-2Ae^{-2As}\Psi(s)+e^{-2As}\Psi'(s) \\
                &\gs \frac{2}{d}e^{-2As}(L(P_tf))^2,
            \end{aligned}
        \end{equation*}
      The  integration of $\left(e^{-2As}\Psi(s)\right)'$ from 0 to $t$ leads to 
        $$e^{-2At}\Psi(t)-\Psi(0)\gs \frac{2}{d}(L(P_tf))^2 \int_0^te^{-2As}\ \mathrm{d}s=\frac{1-e^{-2At}}{Ad}(L(P_tf))^2.$$
       This is the desired estimate, and the proof is finished.
    \end{proof}

    The following lemma is the crucial part of this paper. The proof is inspired by \cite{MN}, in which the existence of cut-off functions
    with gradient and Laplacian estimates is proved on RCD$^*(K,N)$-spaces. With the condition \CD, one can generalize it
    to sub-Riemannian manifolds.

    \begin{lemma}  \label{lem3.5}
        Assume that $(M,\mu)$ satisfies CD$(0,\rho_2,\kappa,d)$ for some  $\rho_2>0$, $\kappa\gs 0$ and $0<d\ls \infty$, and the Hypothesis {\rm (1-4)}.
        Then there exist two constants $\gamma=\gamma(\rho_2,\kappa,d)>0$ and $C_2=C_2(\rho_2,\kappa,d)>0$ such that: for any $x_0\in M$ and  $R\gs 1$, there exists a cut-off function $\varphi_R\in C_0^\infty(M)$ with
                \begin{itemize}
            \item $0\ls\varphi_R\ls 1$ on $M$, $\varphi_R\equiv 1$ on $B_R(x_0)$ and $\mathrm{supp}\ \varphi_R\subset B_{\gamma R}(x_0)$,
            \item $\|\Gamma(\varphi_R)+\left(\frac{R}{2}\right)^2\Gamma^Z(\varphi_R)\|_{L^\infty} \ls \frac{C_2}{R^2}$,
            \item $\|L\varphi_R\|_{L^\infty}\ls\frac{C_2}{R^2}. $
                                \end{itemize}
    \end{lemma}

    \begin{proof} 

        Let $\psi_0\in C_0^\infty(M)$ be the function given in Lemma \ref{lem3.2}, 
        and $ \psi_t=P_t\psi_0$. By Lemma \ref{lem3.4}, we have
        \begin{equation}\label{equ3.2}
            \begin{aligned}
                \Gamma( \psi_t)+\nu\Gamma^Z(\psi_t)+\frac{e^{-2At}-1}{-Ad}(L\psi_t)^2 &\ls P_t\left(\Gamma(\psi_0)+\nu\Gamma^Z(\psi_0)\right)e^{-2At}, 
            \end{aligned}
        \end{equation}
     for all $\nu>0$,   where $A=\min\{\rho_1-\kappa/\nu,\rho_2/\nu\}=-\kappa/\nu\ls 0$, (from $\rho_1=0$, $\rho_2\gs0$ and $\kappa/\nu\gs0$).
        Using the maximum principle and  Lemma \ref{lem3.2}, taking $\nu=(\frac{R}{2})^2,$ we get for all $t>0$   that
     \begin{equation*}
     \begin{split}
     P_t\left(\Gamma(\psi_0)+\nu\Gamma^Z(\psi_0)\right)\ls \|\Gamma(\psi_0)+\nu\Gamma^Z(\psi_0)\|_{\infty} \ls \frac{C_0}{R^2}.
     \end{split}
     \end{equation*}
 Substituting this into the inequality (\ref{equ3.2}) we have  
                 \begin{equation}\label{equ3.3}
                \Gamma( \psi_t)+\nu\Gamma^Z(\psi_t)  \ls \frac{C_0}{R^2}e^{-2At},\qquad\qquad
                  \end{equation} 
                        \begin{equation}\label{equ3.4}
                        \begin{split}
             (\partial_t\psi_t)^2&  = (L\psi_t)^2\\
              &\ls   \frac{C_0}{R^2}\cdot \frac{-Ad}{e^{-2At}-1} \cdot e^{-2At}= \frac{C_0}{R^2}\cdot\frac{-2At\cdot e^{-2At}}{e^{-2At}-1} \cdot \frac{d}{2t}\\&\ls\frac{C_0d}{R^2t}, 
             \end{split}
                        \end{equation} 
   provided  $-2At\ls1$,    where we have used the elementary inequality 
 $$\frac{\tau e^\tau}{e^\tau-1}\ls \frac{1}{1-e^{-1}}\ls 2,\qquad\forall \tau\in(-\infty,1].$$
  Integrating $\partial_s \psi_s$ from 0 to $t$ in (\ref{equ3.4}) we get
        \begin{equation*}
            \begin{aligned}
                |\psi_t-\psi_0| \ls \int_0^t|\partial_s\psi_s|\ds\ls  \frac{2\sqrt{C_0d}}{R}\cdot\sqrt t,
            \end{aligned}
        \end{equation*}
  provided  $-2At\ls1$.     When $t\ls \frac{R^2}{64 C_0d }$,  we have. $ |\psi_t-\psi_0| \ls  1/4$. Therefore, we get
       \begin{equation*}
            \left\{
            \begin{aligned}
                &\psi_{t}(x) \in [\frac{3}{4},1], \quad \mathrm{for} \ x\in B_R(x_0), \\
                &\psi_{t}(x) \in [0,\frac{1}{4}], \quad \mathrm{for} \ x \notin B_{\gamma R}(x_0),
                            \end{aligned}
            \right.
        \end{equation*}
        provided (by recalling $A=-\kappa/\nu$)
        \begin{equation}\label{equ3.5}
        t\ls t_R:=\min\Big\{\frac{1}{2\kappa/\nu},  \frac{R^2}{64 C_0d }\Big\}=\min\Big\{\frac{R^2}{8\kappa},  \frac{R^2}{64 C_0d }\Big\}:=C_1\cdot R^2.
        \end{equation}
        
                We can choose a function $\rho\in C_0^\infty(\R)$ such that $\rho(s)\equiv 1$ on $[\frac{3}{4},1]$, $\rho(s)\equiv 0$ on $[0,\frac{1}{4}]$ 
        and $|\rho'(s)|+|\rho''(s)|\ls C'$ on $\R$. Letting $\varphi_R:=\rho\circ\psi_{t_R}$, we have
        \begin{itemize}
            \item $0\ls\varphi_R\ls 1$ on $M$, $\varphi_R\equiv 1$ on $B_R(x_0)$ and $\mathrm{supp}\ \varphi_R\subset B_{\gamma R}(x_0)$,  
            \item $\|\Gamma(\varphi_R)+\nu\Gamma^Z(\varphi_R)\|_{L^\infty} \ls \frac{C'C_0}{R^2}e^{(2\kappa/\nu)t_R}\ls \frac{C'C_0e}{R^2},\qquad $ (by (\ref{equ3.3}))
        \end{itemize}
        and since $L$ is a diffusion operator, we get from (\ref{equ3.4}) and (\ref{equ3.5}) that 
       \begin{equation*}
       \begin{split}
        \|L\varphi_R\|_\infty&=\|\rho'(\psi_{t_a})L\psi_{t_a}+\rho''(\psi_{t_a})\Gamma(\psi_{t_R})\|_\infty\\
        &\ls\frac{C'\sqrt{C_0d}}{R\cdot\sqrt{ t_R}}+\frac{C'C_0e}{R^2}:=\frac{C'\sqrt{C_0d}}{R\cdot\sqrt{C_1}R}+\frac{C'C_0e}{R^2}.
        \end{split}
        \end{equation*}
        Thus $\varphi_R=\rho\circ\psi_{t_R}$ is the desired cut-off function.
    \end{proof}

\section{Liouville-type theorems (Proof of Theorem \ref{thm1.2} and Theorem \ref{thm1.3}.)}

    In this section, as above, we assume that $(M,\mu)$ is a sub-Riemannian manifold satisfying the the Hypothesis (\ref{a1}-\ref{a4})  
    and  $\mathrm{CD}(0,\rho_2,\kappa,d)$ with respect the sub-Laplacian $L$ and a bi-linear form $\Gamma^Z$, for some $\rho_2\gs 0$, $\kappa\gs 0$ and $d>0$.
    
     In the sequal we will fix a reference point $x_0$ on $M$.
      Let  $\{\varphi_R\}_{R\gs 1}\subset C_0^\infty$ be the cut-off function given in Lemma \ref{lem3.5}.
    Then for $R \gs 1$ and arbitrary $\alpha>1$, we have ${\rm supp}\phi_R\subset B_{\gamma R}$  for some $\gamma=\gamma_{\rho_2,\kappa,d}>0$ and
    \begin{equation} \label{002}
        |L(\varphi_R)^\alpha|=\left|\alpha(\varphi_R)^{\alpha-1}L\varphi_R+\alpha(\alpha-1)(\varphi_R)^{\alpha-2}\Gamma(\varphi_R)\right|
        \ls C(\alpha) R^{-2} (\varphi_R)^{\alpha-2}.
    \end{equation}

We need also the following result about the volume growth for $B_R(x_0)$.
        \begin{proposition}\label{prop2}
        (1) {\rm (Baudoin el. \cite{BBG2014})}
        Let $(M,\mu)$ be a sub-Riemannian manifold satisfying Hypothesis (\ref{a1}-\ref{a4}), and $L$ be a sub-Laplacian on $M$.
        Assume that $M$ satisfies CD$(0,\rho_2,\kappa,d)$ for some  $\rho_2>0$, $\kappa\gs 0$ and $0<d\ls \infty$.
        Then, there exists a constant $C>0$, depending only on $\rho_2,\kappa,d$, for which one have:
   \begin{equation}\label{power-d}
   \mu(B_R(x_0))\ls \frac{C}{p(x_0,x_0,1)}\cdot R^D \qquad \forall \ R\gs 1,
   \end{equation}
        where $D=\left(1+\frac{3\kappa}{2\rho_2}\right)d$, $p(\cdot,\cdot,t)$ is the heat kernel associated with $L$. 
        
   (2) {\rm (Chang el. \cite[Theorem 1.2 and Remark 1.2(1)]{ChaC2014})}   Let $(M, \theta)$ be a $(2n+1)$-dimensional complete Sasakian manifold with nonnegative Tanaka-Webster Ricci tensor, then 
        \begin{equation}\label{power-dn}
        \mu(B_R(x_0))\ls C_n\cdot R^{2n+3} \qquad \forall \ R\gs 1.
        \end{equation}
              \end{proposition}
The volume growth power $D$ in (\ref{power-d}) and $2n+3 $ in (\ref{power-dn}) are not  optimal in general, as in the example of the volume growth of a Heisenberg group $\mathbb H^n$. In this case, in fact it is well-known that $\mu(B_R(x_0))\ls C_n R^{2n+2}$.

   Now we begin to prove our nonexistence results in subelliptic, parabolic and hyperbolic cases by the method of text functions.
    The approach of the proofs originally comes from \cite{Poho2009}.   
    For the sake of simplicity, in following proofs, we use $C$ to denote positive constants,  which may vary from line to line.

    \subsection{Subelliptic case}
        For the inequality
        \begin{equation}  \label{003}
            -Lu\gs |u|^p, \quad p>1,
        \end{equation}
        we define the weak solution:

        \begin{definition}
            $u\in L_{loc}^p(M,\mu)$ is said to be a weak solution of (\ref{003}) if 
            \begin{equation} \label{004}
                \int_M|u|^p\varphi\dmu\ls -\int_MuL\varphi\dmu
            \end{equation}
            for every $0\ls \varphi\in C_0^\infty(M)$.
        \end{definition}

        We have following Liouville type theorem.

        \begin{theorem} \label{thm4.3}
            The elliptic inequality {\rm (\ref{003})} admits no nontrivial weak solution provided there exists a number $D'>0$ such that  $1<p\ls \frac{D'}{D'-2}$ and that
            \begin{equation}\label{equ4.4}
            \limsup_{R\to\infty}\frac{\mu(B_R(x_0))}{R^{D'}}<\infty.
            \end{equation}
        \end{theorem}

        \begin{proof} We denote $B_R(x_0)$ by $B_R$ for simplicity.        
            First we emphasize that $(\varphi_R)^\alpha$, given in Lemma \ref{lem3.5}, is supported in $B_{\gamma R}$ and $L(\varphi_R)^\alpha$ is supported in $B_{\gamma R}\setminus B_R$. 
            Substituting the test function $(\varphi_R)^\alpha$  into (\ref{004}), we have
            \begin{equation} \label{005}
                \begin{aligned}
                    \int_{B_R}|u|^p\dmu&  \ls \int_{B_{\gamma R}}|u|^p(\varphi_R)^\alpha\dmu  
                     \ls \int_{B_{\gamma R}\setminus B_R}|u||L(\varphi_R)^\alpha|\dmu \\
                    & \ls \left(\int_{B_{\gamma R}\setminus B_R}|u|^p(\varphi_R)^\alpha\dmu\right)^\frac{1}{p}\left(\int_{B_{\gamma R}\setminus B_R}\frac{|L(\varphi_R)^\alpha|^{p'}}{(\varphi_R^\alpha)^{p'-1}}\dmu\right)^\frac{1}{p'},
                \end{aligned}
            \end{equation} 
            where in the last inequality we have used H\"older's inequality.
            
            Set $\alpha=\frac{2p}{p-1}$ and apply (\ref{002}). Inequality (\ref{005}) becomes
            \begin{equation} \label{006}
                \begin{aligned}
                    \int_{B_R}|u|^p\dmu& \ls \int_{B_{\gamma R}}|u|^p(\varphi_R)^\alpha\dmu \\
                    &\ls \left(\int_{B_{\gamma R}\setminus B_R}|u|^p(\varphi_R)^\alpha\dmu\right)^\frac{1}{p}R^{-2}(\mu(B_{\gamma R}\setminus B_R))^\frac{p-1}{p}  \\
                    &\ls CR^{-2+\frac{(p-1)D'}{p}}\left(\int_{B_{\gamma R}\setminus B_R}|u|^p(\varphi_R)^\alpha\dmu\right)^\frac{1}{p},
                \end{aligned}
            \end{equation} 
       where we have used the inequality (\ref{equ4.4}).     
            From (\ref{006}), we observe that 
            \begin{equation*}
                \int_{B_{\gamma R}}|u|^p(\varphi_R)^\alpha\dmu \ls CR^{-2+\frac{(p-1)D'}{p}}\left(\int_{B_{\gamma R}}|u|^p(\varphi_R)^\alpha\dmu\right)^\frac{1}{p}
            \end{equation*}
          and remark that  $-2+\frac{(p-1)D'}{p}\ls0$,  which shows that $\int_{B_{\gamma R}}|u|^p\varphi_R^\alpha\dmu$ is uniformly bounded as $R\to \infty$. By the monotone convergence 
            theorem, the right side of inequality (\ref{006}) tends to zero as $R\to \infty$. Then $u\equiv 0$.
        \end{proof}

    \subsection{Parabolic case}
        For the parabolic type inequality:
        \begin{equation} \label{007}
            \left\{
                \begin{aligned}
                    &\partial_t u(t,x)-Lu \gs |u|^p, \quad x\in M,t>0,p>1,    \\
                    &u(0,x)=u_0(x),   \\     
                \end{aligned}
            \right.
        \end{equation}

       we define the weak solution as follows.
        
        \begin{definition}
            A function $u\in L_{loc}^p([0,T)\times M)$ is said to be a weak solution of (\ref{007}) in $[0,T)\times M$ if 
            \begin{equation} \label{008}
                \begin{aligned}
                    \int_0^T \int_M |u(t,x)|^p\varphi(t,x) \dmu(x)\dt &+\int_M u_0(x)\varphi(0,x)\dmu(x) \\
                    &\ls -\int_0^T\int_M u(t,x)(\partial_t+L)\varphi(t,x)\dmu(x)\dt
                \end{aligned}
            \end{equation}
            for any $0\ls\varphi\in C_0^\infty([0,T)\times M)$. If $T=\infty$, we call $u$ a global in time weak solution to (\ref{007}), 
            otherwise we call $u$ a local in time weak solution.  
        \end{definition}

        \begin{theorem} \label{thm4.5}
            Assume that
            \begin{equation} \label{009}
                \liminf_{R\to\infty}\int_{B_R(x_0)}u_0(x)\dmu(x)>0,
            \end{equation}
            Then problem {\rm (\ref{007})} admits no global in time weak solutions provided there exists a constant $D'>0$ such that  $1<p\ls 1+\frac{2}{D'}$ and that {\rm (\ref{equ4.4})} holds.
            
        \end{theorem}

        \begin{proof} 
            We will prove by contradiction. Assume that there exists a global in time weak solution to (\ref{007}).
            
            First, we will construct appropriate test functions.

            Let $\tau\in C_0^\infty(\R)$ be such that $0\ls \tau\ls 1$, $\tau\equiv 1$ on $[0,1]$, supp $\tau\subset [0,2)$, and $|\tau'|+|\tau''|<C$.
            Set $\phi_R(t,x)=\tau(\frac{t}{R^2})\varphi_R(x)$.
            For $\alpha>1$, direct computation shows that 
            $$|\partial_t(\phi_R)^\alpha| \ls CR^{-2}(\phi_R)^{\alpha-1} \ls CR^{-2}(\phi_R)^{\alpha-2},$$
            $$|L(\phi_R)^\alpha| \ls CR^{-2}(\phi_R)^{\alpha-2}.$$

           Substituting the test function    $(\phi_R)^\alpha$ into (\ref{008}), and as in the subelliptic case, let $\alpha=\frac{2p}{p-1}$.
            Notice that
            \begin{equation} \label{supp}
                \begin{aligned}
                    \mathrm{supp}\ (\phi_R)^\alpha &\subset P_{\gamma R}:=[0,2R^2]\times B_{\gamma R}, \\
                    \mathrm{supp}\ \partial_t(\phi_R)^\alpha &\subset \widetilde{P}_{\gamma R}:=[R^2,2R^2]\times B_{\gamma R}, \\
                    \mathrm{supp}\ L(\phi_R)^\alpha &\subset \widehat{P}_{\gamma R}:=[0,R^2]\times (B_{\gamma R}\setminus B_R).
                \end{aligned}
            \end{equation}
            
            Similar to the subelliptic case, we have 
            \begin{equation} \label{010}
                \begin{aligned}
                  &  \quad \iint_{P_{\gamma R}}|u(t,x)|^p\phi_R^\alpha(t,x) \dmu(x)\dt +\int_{B_{\gamma R}} u_0(x)\phi_R^\alpha(0,x)\dmu(x) \\
                    &\ls \iint_{\widetilde{P}_{\gamma R}} |u(t,x)||\partial_t\phi_R^\alpha(t,x)|\dmu(x)\dt  + \iint_{\widehat{P}_{\gamma R}} |u(t,x)||L\phi_R^\alpha(t,x)|\dmu(x)\dt\\
                    &\ls \left(\iint_{\widetilde{P}_{\gamma R}}|u(t,x)|^p\phi_R^\alpha(t,x)\dmu(x)\dt\right)^\frac{1}{p} \left(\iint_{\widetilde{P}_{\gamma R}}\frac{|L\phi_R^\alpha(t,x)|^{p'}}{(\phi_R(t,x))^{\alpha(p'-1)}}\dmu(x)\dt\right)^\frac{1}{p'} \\
                    &\quad \quad + \left(\iint_{\widehat{P}_{\gamma R}}|u(t,x)|^p\phi_R^\alpha(t,x)\dmu(x)\dt\right)^\frac{1}{p} \left(\iint_{\widehat{P}_{\gamma R}}\frac{|\partial_t\phi_R^\alpha(t,x)|^{p'}}{(\phi_R(t,x))^{\alpha(p'-1)}}\dmu(x)\dt\right)^\frac{1}{p'} \\
                    &\ls CR^{-\frac{2}{p}+\frac{(p-1)D'}{p}}\left\{\left(\iint_{\widetilde{P}_{\gamma R}}|u(t,x)|^p\phi_R^\alpha(t,x)\dmu(x)\dt\right)^\frac{1}{p}+\left(\iint_{\widehat{P}_{\gamma R}}|u(t,x)|^p\phi_R^\alpha(t,x)\dmu(x)\dt\right)^\frac{1}{p}\right\}.
                \end{aligned}
            \end{equation}
            This yields that $\iint_{P_{\gamma R}}|u(t,x)|^p\phi_R^\alpha(t,x)\dmu(x)\dt$ is uniformly bounded 
            as $R\to \infty$. By the monotone convergence theorem, 
            the right side of inequality (\ref{010}) tends to zero as $R\to \infty$, which contradicts (\ref{009}).
        \end{proof}

    \subsection{Hyperbolic case}
        For the hyperbolic type inequality
        \begin{equation} \label{011}
            \left\{
                \begin{aligned}
                    &\partial_{tt} u-Lu \gs |u|^p, \quad x\in M,t>0,p>1,    \\
                    &u(0,x)=u_0(x),   \\ 
                    &\partial_t u(0,x)=u_1(x),      
                \end{aligned}
            \right.
        \end{equation}
        we define the weak solution:

        \begin{definition}
            A function $u\in L_{loc}^p([0,T)\times M)$ is said to be a weak solution of (\ref{011}) in $[0,T)\times M$ if 
            \begin{equation} \label{012}
                \begin{aligned}
                    \int_0^T \int_M |u(t,x)|^p\varphi(t,x) \dmu(x)\dt &+\int_M u_1(x)\varphi(0,x)\dmu(x) -\int_Mu_0(x)\partial_t\varphi(0,x)\dmu(x) \\
                    &\ls -\int_0^T\int_M u(t,x)(\partial_{tt}+L)\varphi(t,x)\dmu(x)\dt
                \end{aligned}
            \end{equation}
            for any $0\ls\varphi\in C_0^\infty([0,T)\times M)$.   
        \end{definition}

        \begin{theorem} \label{thm4.7}
            Assume that
            \begin{equation} \label{013}
                \liminf_{R\to\infty}\int_{B_R(x_0)}u_1(x)\dmu(x)>0,
            \end{equation}
            Then hyperbolic inequality {\rm (\ref{011})} admits no global in time weak solution provided there exists a constant $D'>0$ such that $1<p\ls \frac{D'+1}{D'-1}$ and that the inequality {\rm (\ref{equ4.4})} holds.
        \end{theorem}

        \begin{proof}
            Let $\tau$ be as in the proof of Theorem \ref{thm4.5}.
            Letting $\phi_R(t,x)=\tau(\frac{t}{R})\varphi_R(x)$, a similar computation yields 
            \begin{align*}
                |\partial_{tt}(\phi_R(t,x))^\alpha|  &\ls CR^{-2}(\phi_R(t,x))^{\alpha-2}, \\
                |L(\phi_R(t,x))^\alpha| &\ls CR^{-2}(\phi_R(t,x))^{\alpha-2}. 
            \end{align*}
            Clearly, we have
            \begin{align*}
                \mathrm{supp}\ (\phi_R)^\alpha &\subset P_{\gamma R}:=[0,2R]\times B_{\gamma R}, \\
                \mathrm{supp}\ \partial_t(\phi_R)^\alpha &\subset \widetilde{P}_{\gamma R}:=[R,2R]\times B_{\gamma R}, \\
                \mathrm{supp}\ L(\phi_R)^\alpha &\subset \widehat{P}_{\gamma R}:=[0,R]\times (B_{\gamma R}\setminus B_R).
            \end{align*}
            With the same approach as in (\ref{010}), we get
            \begin{align*}
                \quad \iint_{P_{\gamma R}}&|u(t,x)|^p\phi_R^\alpha(t,x) \dmu(x)\dt +\int_{B_{\gamma R}} u_0(x)\phi_R^\alpha(0,x)\dmu(x) \\
                &\ls CR^{-1-\frac{1}{p}+\frac{(p-1)D'}{p}}\left\{\left(\iint_{\widetilde{P}_{\gamma R}}|u(t,x)|^p\phi_R^\alpha(t,x)\dmu(x)\dt\right)^\frac{1}{p}+\left(\iint_{\widehat{P}_{\gamma R}}|u(t,x)|^p\phi_R^\alpha(t,x)\dmu(x)\dt\right)^\frac{1}{p}\right\}.
            \end{align*}
            Then we can prove completely the same as in parabolic case.
        \end{proof}
    Now we are in the place to show the main result.
    \begin{proof}[Proof of Theorem \ref{thm1.2} and Theorem \ref{thm1.3}]  Substituting  Proposition \ref{prop2} into Theorem \ref{thm4.3}, Theorem \ref{thm4.5} and Theorem \ref{thm4.7}, we conclude that both Theorem \ref{thm1.2} and Theorem \ref{thm1.3} hold.
    \end{proof}

\section{Upper bounds of lifespan time (Proof of Theorem \ref{thm1.4}.)} \label{s5}
    For parabolic inequalities with small initial data: 
    \begin{equation} \label{015}
        \left\{
            \begin{aligned}
                &\partial_t u(t,x)-Lu \gs |u|^p, \quad x\in M,t>0,p>1,    \\
                &u(0,x)=\varepsilon u_0(x),  \quad \varepsilon>0 \\     
            \end{aligned}
        \right.
    \end{equation}
    we define the lifespan time of local in time solutions as
    $$T(\varepsilon):=\sup\{T>0\ |\ there\ exists \ a \ weak \ solution \ to \ (\ref{015}) \ in \ [0,T)\times M\}.$$
    We derive a upper bound for $T(\varepsilon)$ in subcritical case.

    \begin{theorem}\label{thm5.1}
        Let $u(t,x):[0,T)\times M\to \R$ be a weak  solution to the parabolic inequality (\ref{015}) with (\ref{009}).
Assume that $1<p< 1+\frac{2}{D}$.
        Then 
        \begin{equation} \label{014}
            T(\varepsilon)\ls C\varepsilon^{-(\frac{1}{p-1}-\frac{D}{2})^{-1}}.
        \end{equation}
    \end{theorem}

    \begin{proof}
        By (\ref{010}) and Young's inequality, we have
        \begin{equation}
            \begin{aligned}
                \iint_{P_{\gamma R}}|u(t,x)|^p\phi_R^\alpha(t,x) \dmu(x)\dt &+\varepsilon\int_{B_{\gamma R}} u_0(x)\phi_R^\alpha(0,x)\dmu(x) \\
                &\ls CR^{-\frac{2}{p}+\frac{(p-1)D}{p}}\left(\iint_{P_{\gamma R}}|u(t,x)|^p\phi_R^\alpha\dmu(x)\dt\right)^\frac{1}{p}  \\
                &\ls \frac{1}{p'}\left(CR^{-\frac{2}{p}+\frac{(p-1)D}{p}}\right)^{p'} + \frac{1}{p}\iint_{P_{\gamma R}}|u(t,x)|^p\phi_R^\alpha\dmu(x)\dt.
            \end{aligned}
        \end{equation}
        By (\ref{009}), we can fix a constant $R_0\gs 1$ large enough such that $\int_{B_{R_0}}u_0(x)\dmu(x)>0$.
        We assume that $T(\varepsilon)\gs \gamma R_0$, since if $T(\varepsilon)<\gamma R_0$, 
        (\ref{014}) is trivially fulfilled. 
        Let 
        $$K:=\int_{B_{R_0}}u_0(x)\phi_R^\alpha(t,x)\dmu(x).$$ 
        Then for any $R\in[\gamma R_0,T(\varepsilon))$,
        $$\int_{B_R}u_0(x)\phi_R^\alpha(t,x)\dmu(x)\gs K.$$
        Then
        \begin{equation*}
            \begin{aligned}
                \varepsilon K &\ls (1-p)\iint_{P_{\gamma R}}|u(t,x)|^p\phi_R^\alpha(t,x) \dmu(x)\dt +\varepsilon\int_{B_{\gamma R}} u_0(x)\phi_R^\alpha(0,x)\dmu(x)  \\
                &\ls \frac{1}{p'}\left(CR^{-\frac{2}{p}+\frac{(p-1)D}{p}}\right)^{p'},
            \end{aligned}
        \end{equation*}
        from which we derive the estimate $R\ls C\varepsilon^{-(\frac{1}{p-1}-\frac{D}{2})^{-1}}$.
        The upper bound of $T(\varepsilon)$ follows immediately.
    \end{proof}

    For the hyperbolic inequality 
    \begin{equation} \label{016}
        \left\{
            \begin{aligned}
                &\partial_{tt} u-Lu \gs |u|^p, \quad x\in M,t>0,p>1,    \\
                &u(0,x)=u_0(x),   \\ 
                &\partial_t  u(0,x)=\sigma u_1(x),  \quad \sigma>0,    
            \end{aligned}
        \right.
    \end{equation}
    we can similarly define the lifespan time of a local in time solution by 
    $$T(\sigma):=\sup\{T>0\ |\ there\ exists \ a \ weak \ solution \ to \ (\ref{016}) \ in \ [0,T)\times M\}.$$
    With the same approach of Theorem \ref{thm5.1}, we get following upper bound of $T(\sigma)$.

    \begin{theorem}\label{thm5.2}
        Let $u(t,x):[0,T)\times M]\to R$ be a weak solution of the hyperbolic inequality (\ref{016}) with (\ref{013}).
   Assume that $1<p<\frac{D+1}{D-1}$. Then 
        $$T(\sigma)\ls C\sigma^{-(\frac{p+1}{p-1}-D)^{-1}}.$$
    \end{theorem}

\begin{proof}[Proof of Theorem \ref{thm1.4}] 
The combination of Theorem \ref{thm5.1} and Theorem \ref{thm5.2} implies Theorem \ref{thm1.4}.
\end{proof}
\begin{remark}Similar as in Theorem \ref{thm1.3}, if we restrict to consider the $(2n+1)$-dimensional CR Sasakian manifold with nonnegative Tanaka-Webster Ricci curvature, then the constant $D$ in Theorem \ref{thm5.1} and Theorem \ref{thm5.2} can be improved to $D_n=2n+3.$
\end{remark}

\section{Upper bounds of lifespan time for general Carnot groups (Proof of Theorem \ref{thm1.7}.)} \label{sec6}
    In this section we specifically study the lifespan of the hyperbolic equation on Carnot group with arbitrary step. Let   $\G$ be a Carnot group of step $r$, and the dimension of the horizontal layer of $\G$ is $d$. Let $Q$ be the homogeneous dimension of $\G$, and let $\mu$ be the Haar measure on $\G$.  We denote $B_R^{\G}(x_0)$  the ball with radius $R>0$ with respect to the homogeneous distance $d_{\G}$. See Sect.\ref{CarnotGroups} for the details.

    \begin{definition} 
        A function $u\in L_{loc}^p([0,T)\times\G)$ is a weak solution of the problem
    \begin{equation} \label{problem}
        \left\{
            \begin{aligned}
                &\partial_{tt}u-\Delta_\G u \gs |u|^p, x\in \G,t>0,    \\
                &u(0,x)=u_0(x),   \\
                &\partial_tu(0,x)=\sigma u_1(x), \sigma>0,     
            \end{aligned}
        \right.
    \end{equation}
in $[0,T)\times\G$ if
        \begin{equation} \label{sseq}
            \begin{aligned}
                \int_0^T\int_{\G}|u(t,x)|^p\varphi(t,x)\dmu(x)\dt &+ \sigma\int_{\G}u_1(x)\varphi(0,x)\dmu(x)
                -\int_{\G}u_0(x)\varphi_t(0,x)\dmu(x) \\
                &\ls-\int_0^T\int_{\G}u(t,x)(\partial_{tt}+\Delta_\G)\varphi(t,x)\dmu(x)\dt
            \end{aligned}
        \end{equation} 
        for any $0\ls \varphi\in C_0^{\infty}([0,T),\times \G)$. If $T=\infty$, we call $u$ a global in time weak solution to (\ref{problem}), 
        otherwise we call $u$ a local in time weak solution. 
    \end{definition}

    As above, we prove a nonexistence result. 

    \begin{proposition} \label{nonexistence}
        Let $1<p\ls \frac{Q+1}{Q-1}$. Assume that
        \begin{equation} \label{u1}
            \liminf_{R\to\infty}\int_{B^{\G}_R(0)}u_1(x)\dmu(x)>0.
        \end{equation}
        Then there exists no global in time weak solution to (\ref{problem}).
    \end{proposition}

    We shall begin with constructing appropriate test functions.

    Let $g\in C_0^\infty([0,\infty)$ be a decreasing bump function such that $g=1$ on $[0,\frac{1}{2}]$, supp $g\subset[0,1)$ and $|g'|+|g''|\ls C$ for 
    some constant $C$. Since $g$ is compactly supported, additionally we have $|g'|+|g''|\ls Cg^{\frac{1}{p}}$
    for some constant $C$ (see \cite{GPCarnot}).
    
    Set
    $$s_R(t,x)=\frac{t^{2r!}+|x|_\G^{2r!}}{R^{2r!}} \quad \mathrm{and} \quad \varphi_R(t,x)=g(s_R(t,x)).$$
    Accords to $s_R(t,x)$, define
    $$D_R=\{(t,x)\in [0,T)\times\G)\ |\ t^{2r!}+|x|_\G^{2r!}<R^{2r!}\}.$$
    Then $\varphi_R$ is in $C_0^\infty([0,T)\times\G)$ and supp $\varphi\subset D_R\subset[0,R)\times B^{\G}_R(0)$.
    
    \begin{proof}[Proof of proposition \ref{nonexistence}]
        Assume that there exists a global in time weak solution to (\ref{problem}). We will prove by contradiction.

        Through a simple calculation, we get 
        \begin{equation}
            \begin{aligned} \label{e2}
                |\partial_{tt}\varphi_R(t,x)| &\ls \left(2r!\frac{t^{2r!-1}}{R^{2r!}}\right)^2\left|g''\circ s_R(t,x)\right| + \left(2r!(2r!-1)\frac{t^{2r!-2}}{R^{2r!}}\right)\left|g''\circ s_R(t,x)\right| \\
                &\ls C(r)R^{-2}(g\circ s_R(t,x))^\frac{1}{p} \\
                &=C(r)R^{-2}(\varphi_R(t,x))^\frac{1}{p}.
            \end{aligned}
        \end{equation}
        Remark that 
        $$s_R(t,x)=(R^{-1}t)^{2r!}+|\delta_{R^{-1}}(x)|_\G^{2r!}.$$
        Since, for $1\ls k\ls m$ and $\lambda>0$, $X_k$ is $\delta_\lambda$ homogeneous of degree 1, we have
        \begin{equation*}
            \begin{aligned}
                X_k^2(\varphi_R)(t,x)&=(X_k(s_R))^2(g''\circ s_R(t,x))+X_k^2(s_R)(g'\circ s_R(t,x)) \\ 
                &=R^{-2}\left(X_k(|\delta_{R^{-1}}(x)|_\G^{2r!})\right)^2(g''\circ s_R(t,x))+R^{-2}\left(X_k^2|\delta_{R^{-1}}(x)|_\G^{2r!}\right)(g'\circ s_R(t,x)).
            \end{aligned}
        \end{equation*}
        Note that on the compact support of $g\circ s_R$, we have $\delta_{R^{-1}}(x)\in D_1$, and
        $X_k(|\cdot|_\G^{2r!})$, $X_k^2|\cdot|_\G^{2r!}$ are smooth functions, so is bounded on $\overline{D_1}$.
        Thus we get 
        \begin{equation} \label{e3}
            |X_k^2(\varphi_R)(t,x)| \ls CR^{-2}(\varphi_R(t,x))^{\frac{1}{p}}.
        \end{equation}
        
        Combine (\ref{e2}) and (\ref{e3}), we obtain the estimate 
        $$|(\partial_{tt}+\Delta_\G)\varphi_R(t,x)| \ls CR^{-2}(\varphi_R(t,x))^{\frac{1}{p}}.$$

        Let $\varphi_R$ be the test function in (\ref{sseq}). Note that $\partial_t\varphi_R(0,x)=0$. 
        Since $\varphi_R$ is supported in $D_R$, and clearly $(\partial_{tt}+\Delta_\G)\varphi_R(t,x)$
        is supported in $D_R\setminus D_{\frac{R}{2}}$, we obtain
        \begin{equation} \label{mainineq}
            \begin{aligned}
                \iint_{D_R}&|u(t,x)|^p\varphi_R(t,x)\dmu(x)\dt + \sigma\int_{B_R}u_1(x)\varphi_R(0,x)\dmu(x)  \\
                &\ls \iint_{D_R\setminus D_{\frac{R}{2}}}|u(t,x)||(\partial_{tt}+\Delta_\G)\varphi_R(t,x)|\dmu(x)\dt \\
                &\ls \left(\iint_{D_R\setminus D_{\frac{R}{2}}}|u(t,x)|^p\varphi_R(t,x)\dmu(x)\dt\right)^{\frac{1}{p}} 
                \left(\iint_{D_R\setminus D_{\frac{R}{2}}}\frac{|(\partial_{tt}+\Delta_\G)\varphi_R(t,x)|^{p'}}{(\varphi_R(t,x))^{p'-1}}\dmu(x)\dt\right)^{\frac{1}{p'}}.
            \end{aligned}
        \end{equation}

        By $1<p\ls \frac{Q+1}{Q-1}$, the the last term 
        \begin{align*}
            \left(\iint_{D_R\setminus D_{\frac{R}{2}}}\frac{|(\partial_{tt}+\Delta_\G)\varphi_R(t,x)|^{p'}}{(\varphi_R(t,x))^{p'-1}}\dmu(x)\dt\right)^{\frac{1}{p'}} 
            &\ls CR^{-2}\int_0^R\mu (B_R(0))\dt \\
            &\ls CR^{-2+\frac{Q+1}{p'}} = CR^{Q-1-\frac{Q+1}{p}}.
        \end{align*}
        is bounded.

        Similar to previous proofs, $\int_0^R\int_{B^{\G}_R(0)}|u(t,x)|^p\varphi_R(t,x)\dmu(x)\dt$ is uniformly bounded as $R\to \infty$. By the 
        monotone convergence theorem, the right side of inequality (\ref{mainineq}) tends to zero as $R\to \infty$, which contradicts (\ref{u1}).
        The proof is completed.
    \end{proof}


    Now we are the place to prove the main result in this section. We will adapt the approach in \cite{Ikeda,GPCarnot}.

    \begin{theorem} \label{mthm}
        Let $1<p\ls \frac{Q+1}{Q-1}$ and $u_1$ be as in (\ref{u1}).
        Then there exists $\sigma_0>0$ such that for any $\sigma\in (0,\sigma_0]$, there holds
        \begin{equation}
            T(\sigma)\ls
            \begin{cases}
                C\sigma^{-({\frac{p+1}{p-1}-Q)}^{-1}} & if\ p\in(1,\frac{Q+1}{Q-1}), \\
                \exp(C\sigma^{-(p-1)}) & if\ p=\frac{Q+1}{Q-1}.
            \end{cases}
        \end{equation}
        Where $C$ is a positive constant independent of $\sigma$.
    \end{theorem}

    \begin{proof}[Proof of Theorem \ref{mthm}]
        Fix $R_0$ such that
        $$\int_{B^{\G}_{R_0}(0)}u_1(x)\varphi_{R_0}(0,x)\dmu(x) >0.$$
        Function g is decreasing, so for $R\gs R_0$, we have $\varphi_{R_0} \ls \varphi_R$. Let
        $$K:=\int_{B^{\G}_{R_0}(0)}u_1(x)\varphi_{R_0}(0,x)\dmu(x).$$
        Then for $R\in[R_0,T(\sigma))$, 
        $$\int_{B^{\G}_R(0)}u_1(x)\varphi_R(0,x)\dmu(x) \gs K.$$
        We wssume that $T(\sigma)>R_0$. Since if $T(\sigma)\ls R_0$, theorem \ref{mthm} is trivially fulfilled.

        From the proof of Propositon \ref{nonexistence}, we know that for any $R\in[R_0,T(\sigma))$, we have
        \begin{equation} 
            \begin{aligned} \label{e4}
                \iint_{D_R}&|u(t,x)|^p\varphi_R(t,x)\dmu(x)\dt+\sigma K\\
                 &\ls
                \iint_{D_R}|u(t,x)|^p\varphi_R(t,x)\dmu(x)\dt + \sigma\int_{B^{\G}_R(0)}u_1(x)\varphi_R(0,x)\dmu(x) \\
                &\ls CR^{Q-1-\frac{Q+1}{p}}\left(\iint_{D_R \setminus D_{\frac{R}{2}}}|u(t,x)|^p\varphi_R(t,x)\dmu(x)\dt\right)^{\frac{1}{p}}.
            \end{aligned}
        \end{equation}

        Define
        \begin{equation*}
            g^*(s):=\begin{cases}
                0, & if \ s\in[0,\frac{1}{2}),  \\
                g(s), & if \ s\in[\frac{1}{2}, \infty],
            \end{cases}
        \end{equation*}
        and
        $$\varphi_R^*(x,t):=g^*(s_R(x,t)).$$
        Then $\varphi_R^*\in L^\infty([0,T)\times\G)$. And (\ref{e4}) can be written as
        \begin{equation} \label{e5}
            \iint_{D_R}|u(t,x)|^p\varphi_R(t,x)\dmu(x)\dt+\sigma K 
            \ls CR^{Q-1-\frac{Q+1}{p}}\left(\iint_{D_R}|u(t,x)|^p\varphi_R^*(t,x)\dmu(x)\dt\right)^{\frac{1}{p}}.
        \end{equation}

        Next we introduce a lemma analogous to (\cite{GPCarnot}).
        \begin{lemma} \label{lem}
            Let $f=f(s)$ be a measurable function such that $f(s)=0$ for $s\in [0,\frac{1}{2}]\cup [0,\infty)$ and $f(s)$ is a decreasing
            function for $s>1$. Then for any $R>0$, $A>0$ and $h>0$, we have 
            \begin{equation} \label{lemeq}
                \int_0^Rf\left(\frac{A}{\rho^h}\right)\ \frac{\drho}{\rho} \ls \frac{\ln 2}{h}f\left(\frac{A}{R^h}\right).
            \end{equation}
        \end{lemma}
        \begin{proof}
            If $\frac{A}{R^h}>1$, (\ref{lemeq}) is trivial since the left side of (\ref{lemeq}) is identically 0.
            Else for $\frac{A}{R^h} \ls 1$, 
            \begin{equation*}
                \begin{aligned}
                    \int_0^Rf\left(\frac{A}{\rho^h}\right)\ \frac{\drho}{\rho}
                    &= \int_{[0,R]\cap[A^{\frac{1}{h}},(2A)^{\frac{1}{h}}]}f\left(\frac{A}{\rho^h}\right)\ \frac{\drho}{\rho} \\
                    &\ls f\left(\frac{A}{R^h}\right)\int_{A^{\frac{1}{h}}}^{(2A)^{\frac{1}{h}}}\ \frac{\drho}{\rho}
                    =\frac{\ln2}{h}f\left(\frac{A}{R^h}\right).
                \end{aligned}
            \end{equation*}
        \end{proof}

        Let us continue the proof of the Theorem \ref{mthm}.

        Denote
        $$X(\rho):=\iint_{D_\rho}|u(t,x)|^p\varphi_\rho(t,x)\dmu(x)\dt,$$
        $$Y(\rho):=\iint_{D_\rho}|u(t,x)|^p\varphi_\rho^*(t,x)\dmu(x)\dt,$$
        By setting $f=g^*$, $A=s_1(t,x)$ and $h=2r!$ in Lemma (\ref{lem}), we get
        $$\frac{2r!}{\ln 2}\int_0^RY(\rho)\ \frac{\drho}{\rho} \ls X(R),$$
        Combine with (\ref{e5}), we have
        \begin{equation}
            \frac{2r!}{\ln 2}\int_0^RY(\rho)\ \frac{\drho}{\rho}+\sigma K \ls X(R)+\sigma K \ls CR^{Q-1-\frac{Q+1}{p}}(Y(R))^{\frac{1}{p}}.
        \end{equation}
        Then 
        \begin{equation} \label{e6}
            \begin{aligned}
                \frac{\mathrm{d}}{\mathrm{d}R}\left(\frac{2r!}{\ln 2}\int_0^RY(\rho)\ \frac{\drho}{\rho}+\sigma K\right)^{1-p}
                &=(1-p)^{-1}\frac{2r!}{\ln 2}\left(\frac{2r!}{\ln 2}\int_0^RY(\rho)\ \frac{\drho}{\rho}+\sigma K\right)^{-p}Y(R)R^{-1}\\
                &\ls -CR^{-p(Q-1)+Q}.
            \end{aligned}
        \end{equation}
        Integrate (\ref{e6}) from $R_0$ to $T(\sigma)$. We can choose a constant $\sigma_0$ small enough such that for every $\sigma\in (0,\sigma_0]$, we have
        \begin{equation}
            -C\left.\left(\frac{2r!}{\ln 2}\int_0^RY(\rho)\ \frac{\drho}{\rho}+\sigma K\right)^{1-p}\right|_{R_0}^{T(\sigma)} \gs
            \begin{cases}
                \ln \left(\frac{T(\sigma)}{R_0}\right),\quad &if\ p=\frac{Q+1}{Q-1}, \\
                (T(\sigma))^{-p(Q-1)+Q+1},\quad &if\ 1<p<\frac{Q+1}{Q-1}.
            \end{cases}
        \end{equation}
        By simple calculation we see
        $$-\left.\left(\frac{2r!}{\ln 2}\int_0^RY(\rho)\ \frac{\drho}{\rho}+\sigma K\right)^{1-p}\right|_{R_0}^{T(\sigma)}\ls -C\sigma^{-(p-1)}.$$
        Finally we get
        \begin{equation}
            T(\sigma)\ls
            \begin{cases}
                C\sigma^{-({\frac{p+1}{p-1}-Q)}^{-1}} & if\ p\in(1,\frac{Q+1}{Q-1}), \\
                \exp(C\sigma^{-(p-1)}) & if\ p=\frac{Q+1}{Q-1},
            \end{cases}
        \end{equation}
        which is exactly the bound we state in Theorem \ref{mthm}.
    \end{proof}
        
  At last, we complete the proof of Theorem \ref{thm1.7} as follows.
  \begin{proof}[Proof of Theorem \ref{thm1.7}] It comes from the combination of Theoerm \ref{mthm} and the fact the homogeneous distance $d_{\G}$ is equivalent to the Carnot-Carath\'eodory distance.
  \end{proof} 
  
\bibliographystyle{amsplain}

\providecommand{\href}[2]{#2}

\end{document}